\documentclass[11pt]{amsart}%
\usepackage{graphicx,amscd,color,amsmath,amsfonts,amssymb,geometry,hyperref,soul,eurosym}
\usepackage{amsmath}
\usepackage{amsfonts}
\usepackage{amssymb}
\usepackage{graphicx}%
\providecommand{\U}[1]{\protect\rule{.1in}{.1in}}
\newtheorem{theorem}{Theorem}
\theoremstyle{plain}

\newtheorem{corollary}{Corollary}

\numberwithin{equation}{section}
\begin{document}
\title[Critical Hardy--Littlewood inequality for multilinear forms]{Critical Hardy--Littlewood inequality for multilinear forms}
\author[Djair Paulino ]{Djair Paulino}
\address{ \\
\indent\\
\indent}
\thanks{Partially supported by Capes.}
\thanks{MSC2010: }

\begin{abstract}
The Hardy--Littlewood inequalities for $m$-linear forms on $\ell_{p}$ spaces
are known just for $p>m$. The critical case $p=m$ was overlooked for obvious
technical reasons and, up to now, the only known estimate is the trivial one.
In this paper we deal with this critical case of the Hardy--Littlewood
inequality. More precisely, for all positive integers $m\geq2$ we have%
\[
\sup_{j_{1}}\left(  \sum_{j_{2}=1}^{n}\left(  .....\left(  \sum_{j_{m}=1}%
^{n}\left\vert T\left(  e_{j_{1}},\dots,e_{j_{m}}\right)  \right\vert ^{s_{m}%
}\right)  ^{\frac{1}{s_{m}}\cdot s_{m-1}}.....\right)  ^{\frac{1}{s_{3}}s_{2}%
}\right)  ^{\frac{1}{s_{2}}}\leq2^{\frac{m-2}{2}}\left\Vert T\right\Vert
\]
for all $m$--linear forms $T:\ell_{m}^{n}\times\cdots\times\ell_{m}%
^{n}\rightarrow\mathbb{K}=\mathbb{R}$ or $\mathbb{C}$ with $s_{k}%
=\frac{2m(m-1)}{m+mk-2k}$ for all $k=2,....,m$ and for all positive integers
$n$. As a corollary, for the classical case of bilinear forms investigated by
Hardy and Littlewood in 1934 our result is sharp in a strong sense (both
exponents and constants are optimal for real and complex scalars).

\end{abstract}
\maketitle
\dedicatory{\textit{To the memory of Joe Diestel}}


\section{Introduction}

The Hardy--Littlewood inequalities for $m$--linear forms on $\ell_{p}$ spaces
are valid for $p>m;$ they are natural versions of the original inequalities of
Hardy and Littlewood for bilinear forms (\cite{hardy}). When $\ell_{p}$ is
replaced by $c_{0}$ we have the famous Littlewood's $4/3$ inequality
(\cite{43}) when $m=2$ and the Bohnenblust--Hille inequality \cite{bh} for the
general case (for a recent panorama of the subject see \cite{pt} the
references therein and also \cite{montanaro} for applications in Physics).

Since 2014 these inequalities have began to be explored in the anisotropic
setting and new subtle information was shed to light. In 2017 a series of
papers (in chronological order, \cite{p2, bayart, rezende}) proved new
inclusion theorems for multiple summing operators, of fundamental importance
to a better understanding of these inequalities (see, for instance,
\cite{cavalcante, Nunes} and the references therein).

In this paper we present a Hardy--Littlewood inequality for the critical case
$p=m.$ Up to now the only known result for this case was the absolutely
trivial one (we could say almost tautological): for any positive integer $m$
we have%
\[
\sup_{j_{1},...,j_{m}}\left\vert T(e_{j_{1}},\cdots,e_{j_{m}})\right\vert
\leq\left\Vert T\right\Vert ,
\]
for all $m$--linear forms $T:\ell_{m}^{n}\times\cdots\times\ell_{m}%
^{n}\rightarrow\mathbb{K}=\mathbb{R}$ or $\mathbb{C}$. We prove the following
estimates for the critical case:

\begin{theorem}
\label{444}For all $m\geq2$ we have%
\begin{equation}
\sup_{j_{1}}\left(  \sum_{j_{2}=1}^{n}\left(  .....\left(  \sum_{j_{m}=1}%
^{n}\left\vert T\left(  e_{j_{1}},\dots,e_{j_{m}}\right)  \right\vert ^{s_{m}%
}\right)  ^{\frac{1}{s_{m}}s_{m-1}}.....\right)  ^{\frac{1}{s_{3}}s_{2}%
}\right)  ^{\frac{1}{s_{2}}}\leq2^{\frac{m-1}{2}}\left\Vert T\right\Vert
\label{71}%
\end{equation}
for all $m$--linear forms $T:\ell_{m}^{n}\times\cdots\times\ell_{m}%
^{n}\rightarrow\mathbb{K},$ with
\[
s_{k}=\frac{2m(m-1)}{m+mk-2k}%
\]
for all $k=2,....,m.$ Moreover, $s_{1}=\infty$ and $s_{2}=m$ are sharp and,
for $m>2$ the optimal exponents $s_{k}$ satisfying (\ref{71}) fulfill%
\[
s_{k}\geq\frac{m}{k-1}.
\]

\end{theorem}

The paper is organized as follows. In Section 2 we prove an improvement of a
recent result of Albuquerque and Rezende \cite{rezende}, of independent
interest. In Sections 3 and 4 we prove Theorem \ref{444} and, as a
consequence, in Section 5 we complete results of Pellegrino, Santos, Serrano,
Teixeira \cite{p2} on the Hardy--Littlewood inequalities for the bilinear case
(original case considered by Hardy and Littlewood). In Section 6 we
investigate the dependence on $n$ when we deal with non-admissible exponents.

\section{Preliminaries: improvement of the Albuquerque-Rezende Inclusion
Theorem}

In this section $X,Y$ shall stand for Banach spaces over the scalar field
$\mathbb{K}$ of real or complex numbers. The topological dual of $X$ and its
closed unit ball are denoted by $X^{\ast}$ and $B_{X^{^{\ast}}}$,
respectively. For $r,p\geq1,$ a linear operator $T:X\rightarrow Y$ is said
absolutely $(r;p)$-summing if there exists a constant $C>0$ such that%
\[
\left(  \sum_{j=1}^{n}\left\Vert T\left(  x_{j}\right)  \right\Vert
^{r}\right)  ^{\frac{1}{r}}\leq C\left\Vert (x_{j})_{j=1}^{n}\right\Vert
_{w,p},
\]
for all positive integers $n$, where%
\[
\left\Vert (x_{j})_{j=1}^{n}\right\Vert _{w,p}:=\sup_{\varphi\in B_{X^{\ast}}%
}\left(  \sum_{j=1}^{n}\left\vert \varphi\left(  x_{j}\right)  \right\vert
^{p}\right)  ^{\frac{1}{p}}.
\]
One of the most relevant extensions of absolutely summing operators to the
multilinear setting is the notion of multiple summing operators, introduced,
independently by M.C.\ Matos and D. P\'{e}rez-Garc\'{\i}a (see \cite{matos,
perez}; for other related concepts we refer to \cite{cp}). It is convenient to
recall the anisotropic version of multiple summing operators (the basics of
this theory are sketched in \cite{araujo}): For $\mathbf{r\in\lbrack
1,+\infty)}^{m}\mathbf{,p\in\lbrack1,+\infty]}^{m},$ a $m$-linear operator
$T:X_{1}\times\cdots\times X_{m}\rightarrow Y$ is multiple $\left(
\mathbf{r,p}\right)  $-summing if there is a constant $C>0$ such that for all
sequences $x^{k}:=\left(  x_{j}^{k}\right)  _{j\in\mathbb{N}},k=1,...,m,$ we
have%
\[
\left\Vert T\left(  x_{\mathbf{j}}\right)  \right\Vert _{r}:=\left(
\sum_{j_{1}=1}^{n}\left(  .....\left(  \sum_{j_{m}=1}^{n}\left\vert T\left(
x_{\mathbf{j}}\right)  \right\vert ^{r_{m}}\right)  ^{\frac{r_{m}-1}{r_{m}}%
}.....\right)  ^{\frac{r_{1}}{r_{2}}}\right)  ^{\frac{1}{r_{1}}}\leq C%
{\displaystyle\prod\limits_{k=1}^{m}}
\left\Vert (x_{j}^{(k)})_{j=1}^{n}\right\Vert _{w,p_{k}},
\]
for all positive integers $n$, where $T\left(  x_{\mathbf{j}}\right)
:=T\left(  x_{j_{1}}^{1},...,x_{j_{m}}^{m}\right)  .$ When $r_{i}=\infty$ we
consider the $\sup$ norm replacing the respective $\ell_{r_{i}}$ norm. The
class of all multiple $\left(  \mathbf{r,p}\right)  $-summing operators is a
Banach space with the norm defined by the infimum of all previous constants
$C>0.$ Following the usual conventions, the space of all such operators is
denoted by $\Pi_{(\mathbf{r};\mathbf{p})}^{m}(X_{1},...,X_{m},Y).$ When
$r_{1}=\cdots=r_{m}=r,$ we simply write $\left(  r;\mathbf{p}\right)  $. For
$\mathbf{p\in\lbrack1,+\infty)}^{m}$ and each $k\in\{1,..,m\},$ we also follow
the usual notation and define
\begin{align*}
\left\vert \frac{1}{\mathbf{p}}\right\vert _{\geq k}  &  :=\frac{1}{p_{k}%
}+\cdots+\frac{1}{p_{m}},\\
\left\vert \frac{1}{\mathbf{p}}\right\vert  &  :=\left\vert \frac
{1}{\mathbf{p}}\right\vert _{\geq1}.
\end{align*}
Recently Albuquerque and Rezende \cite[Theorem 3]{rezende} have proved the
following result:

\begin{theorem}
[Albuquerque and Rezende (\cite{rezende})]\label{rrr} Let $m$ be a positive
integer, $r\geq1,$ and $\mathbf{s},\mathbf{p,q\in\lbrack1,+\infty)}^{m}$ are
such that $q_{k}\geq p_{k},$ for $k=1,...,m$ and
\[
\frac{1}{r}-\left\vert \frac{1}{\mathbf{p}}\right\vert +\left\vert \frac
{1}{\mathbf{q}}\right\vert >0.
\]
Then%
\[
\Pi_{(r;\mathbf{p})}^{m}(X_{1},...,X_{m},Y)\subset\Pi_{(\mathbf{s}%
;\mathbf{q})}^{m}(X_{1},...,X_{m},Y),
\]
for any Banach spaces $X_{1},...,X_{m}$, with%
\[
\frac{1}{s_{k}}-\left\vert \frac{1}{\mathbf{q}}\right\vert _{\geq k}=\frac
{1}{r}-\left\vert \frac{1}{\mathbf{p}}\right\vert _{\geq k}%
\]
for each $k\in\{1,...,m\},$ and the inclusion operator has norm $1$.
\end{theorem}

The proof is technical, as expected, and one of the main ingredients is the
following inequality of Minkowski:

\textbf{Minkowski Inequality:} For any $0<p\leq q<\infty$ and for any scalar
matrix $\left(  a_{ij}\right)  _{i,j=1}^{n}$ we have%
\begin{equation}
\left(  \sum\limits_{i=1}^{n}\left(  \sum\limits_{j=1}^{n}\left\vert
a_{ij}\right\vert ^{p}\right)  ^{\frac{q}{p}}\right)  ^{\frac{1}{q}}%
\leq\left(  \sum\limits_{j=1}^{n}\left(  \sum\limits_{i=1}^{n}\left\vert
a_{ij}\right\vert ^{q}\right)  ^{\frac{p}{q}}\right)  ^{\frac{1}{p}}
\label{mm}%
\end{equation}
However, it is simple to verify that this inequality is also valid for
$q=\infty$, replacing the $\ell_{q}$ norm by the $\sup$ norm. Using this fact
and mimicking the proof of Theorem \ref{rrr}, we can provide the following improvement:

\begin{theorem}
[Inclusion Theorem extended ]\label{990}Let $m$ be a positive integer,
$r\geq1,$ and $\mathbf{p,q,s\in\lbrack1,+\infty]}^{m}.$ If $q_{k}\geq p_{k}$
for $k=1,...,m$ and%
\begin{equation}
\frac{1}{r}-\left\vert \frac{1}{\mathbf{p}}\right\vert +\left\vert \frac
{1}{\mathbf{q}}\right\vert >0 \label{ar1}%
\end{equation}
or $q_{k}\geq p_{k}$ for $k=2,...,m$ and $q_{1}>p_{1}$ and
\begin{equation}
\frac{1}{r}-\left\vert \frac{1}{\mathbf{p}}\right\vert +\left\vert \frac
{1}{\mathbf{q}}\right\vert \geq0 \label{ar2}%
\end{equation}
then%
\[
\Pi_{(r;\mathbf{p})}^{m}(X_{1},...,X_{m},Y)\subset\Pi_{(\mathbf{s}%
;\mathbf{q})}^{m}(X_{1},...,X_{m},Y),
\]
for any Banach spaces $X_{1},...,X_{m}$, with%
\[
\frac{1}{s_{k}}-\left\vert \frac{1}{\mathbf{q}}\right\vert _{\geq k}=\frac
{1}{r}-\left\vert \frac{1}{\mathbf{p}}\right\vert _{\geq k}%
\]
for each $k\in\{1,...,m\},$ and the inclusion operator has norm $1$.
\end{theorem}

\begin{proof}
If we suppose $q_{k}\geq p_{k}$ for $k=1,...,m$ and%
\[
\frac{1}{r}-\left\vert \frac{1}{\mathbf{p}}\right\vert +\left\vert \frac
{1}{\mathbf{q}}\right\vert >0
\]
we just need to follow the proof of \cite{rezende}, observing that there is no
technical problem in considering $\mathbf{p,q,s\in\lbrack1,+\infty]}^{m}$ and
observing that (\ref{mm}) is valid for $q=\infty.$ So, let us suppose
$q_{k}\geq p_{k}$ for $k=2,...,m$ and $q_{1}>p_{1}$ and%
\[
\frac{1}{r}-\left\vert \frac{1}{\mathbf{p}}\right\vert +\left\vert \frac
{1}{\mathbf{q}}\right\vert \geq0.
\]
For the obvious reasons we just need to deal with the case%
\[
\frac{1}{r}-\left\vert \frac{1}{\mathbf{p}}\right\vert +\left\vert \frac
{1}{\mathbf{q}}\right\vert =0.
\]

Let
\[
T\in\Pi_{(r;\mathbf{p})}^{m}(X_{1},...,X_{m},Y).
\]
By the hypothesis and by Theorem \cite{rezende}, we conclude that
\[
T\in\Pi^{m}_{(\mathbf{\overline{s}};p_{1},q_{2},...,q_{m})}(X_{1}%
,...,X_{m},Y),
\]
with $\overline{s}_{k}=s_{k}$, for $k=2,...,m$ and $\overline{s}_{1}=s_{2},$
that is, exists $C>0$ such that
\begin{align*}
&  \left(  \sum_{j_{1}=1}^{n}\left(  \sum_{j_{2}=1}^{n}\left(  .....\left(
\sum_{j_{m}=1}^{n}\left\vert T\left(  x_{j_{1}}^{(1)},x_{j_{2}}^{(2)}%
,\dots,x_{j_{m}}^{(m)}\right)  \right\vert ^{s_{m}}\right)  ^{\frac{1}{s_{m}%
}s_{m-1}}.....\right)  ^{\frac{1}{s_{3}}s_{2}}\right)  ^{\frac{1}{s_{2}}s_{2}%
}\right)  ^{\frac{1}{s_{2}}}\\
&  \leq C%
{\textstyle\prod\limits_{k=2}^{m}}
\left\Vert \left(  x_{j_{k}}^{(k)}\right)  _{j_{k}=1}^{n}\right\Vert
_{w,q_{k}}\cdot\Vert(x_{j_{1}}^{(1)})_{j_{1}=1}^{n}\Vert_{w,p_{1}}%
\end{align*}

Fixed, $(x_{j_{2}}^{(2)})_{j_{2}=1}^{n},...,(x_{j_{m}}^{(m)})_{j_{m}=1}^{n}$,
consider the operator $T_{1}:X_{1}\longrightarrow\ell_{(s_{2},...,s_{m})}(Y)$
defined by $T_{1}(x)=(T(x,x_{j_{2}}^{(2)},...,x_{j_{m}}^{(m)}))_{j_{2}%
,...,j_{m}}$. By the inequality above,
\[
T_{1}\in\Pi_{(s_{2};p_{1})}(X_{1},\ell_{(s_{2},...,s_{m})}(Y)).
\]
And, by the classical Inclusion Theorem, $T_{1}\in\Pi_{(t;q_{1})}(X_{1}%
,\ell_{(s_{2},...,s_{m})}(Y))$, with $t\geq s_{2}$ and
\begin{equation}
\frac{1}{p_{1}}-\frac{1}{s_{2}}\leq\frac{1}{q_{1}}-\frac{1}{t}. \label{131}%
\end{equation}
As $t=s_{1}$ satisfies the inequality (\ref{131}), we have $T_{1}\in
\Pi_{(s_{1};q_{1})}(X_{1},\ell_{(s_{2},...,s_{m})}(Y))$ and, therefore,
\[
\left(  \sum_{j_{1}=1}^{n}\Vert T_{1}(x_{j_{1}}^{(1)})\Vert^{s_{1}}\right)
^{\frac{1}{s_{1}}}\leq C_{1}\Vert(x_{j_{1}}^{(1)})_{j_{1}=1}^{n}\Vert
_{w,q_{1}},
\]
with $C_{1}=C\left\Vert \left(  x_{j_{k}}^{(k)}\right)  \right\Vert _{w,q_{k}%
}$, so $T\in\Pi_{(\mathbf{s};\mathbf{q})}^{m}(X_{1},...,X_{m},Y).$
\end{proof}

As a corollary we have:

\begin{corollary}
For all $m\geq2$ we have%
\[
\sup_{j_{1}}\left(  \sum_{j_{2}=1}^{n}\left(  .....\left(  \sum_{j_{m}=1}%
^{n}\left\vert T\left(  e_{j_{1}},\dots,e_{j_{m}}\right)  \right\vert ^{s_{m}%
}\right)  ^{\frac{1}{s_{m}}s_{m-1}}.....\right)  ^{\frac{1}{s_{3}}s_{2}%
}\right)  ^{\frac{1}{s_{2}}}\leq2^{\frac{m-1}{2}}\left\Vert T\right\Vert
\]
for all $m$--linear forms $T:\ell_{m}^{n}\times\cdots\times\ell_{m}%
^{n}\rightarrow\mathbb{K},$ with
\begin{equation}
s_{k}=\frac{2m}{k} \label{v665}%
\end{equation}
for all $k=2,....,m.$
\end{corollary}

In the next section we prove the first part of Theorem \ref{444}, which
improves the estimates (\ref{v665}).

\section{The proof of the theorem \ref{444}: part 1}

From now on, if $p\in\lbrack1,\infty]$, as usual, we define $p^{\ast}$ by%
\[
\frac{1}{p^{\ast}}=1-\frac{1}{p}%
\]
and consider$\frac{1}{\infty}=0$. Using, as in \cite{cavalcante}, that for all
Banach spaces $X_{1},...,X_{m}$ and all $m$-linear forms from $X_{1}%
\times\cdots\times X_{m}$ to $\mathbb{K}$ are multiple $\left(  2;\left(
2m\right)  ^{\ast},....,\left(  2m\right)  ^{\ast}\right)  $-summing, we
already know that from the original version of Theorem 2 we can conclude
nothing for multiple $\left(  \mathbf{s};m^{\ast},....,m^{\ast}\right)
$-summing operators because in this case%
\[
\frac{1}{r}-\left\vert \frac{1}{\mathbf{p}}\right\vert +\left\vert \frac
{1}{\mathbf{q}}\right\vert =0.
\]
Let us fix one of the variables (say the first one) and work with $\left(
m-1\right)  $-multilinear forms from $X_{2}\times\cdots\times X_{m}$ to
$\mathbb{K}$. From the Hardy--Littlewood inequality for $\left(  m-1\right)
$-linear forms we know that for any fixed vector $a_{1}\in X_{1}$ we have that
$T(a_{1},\cdot,...,\cdot)$ is $\left(  2;\left(  2(m-1)\right)  ^{\ast
},....,\left(  2(m-1)\right)  ^{\ast}\right)  $-summing. So,
\[
\sup_{\left\Vert x_{j_{1}}^{(1)}\right\Vert \leq1}\left(  \sum_{j_{2}%
,...,j_{m}=1}^{n}\left\vert T\left(  x_{j_{1}}^{(1)},\dots,x_{j_{m}}%
^{(m)}\right)  \right\vert ^{2}\right)  ^{\frac{1}{2}}\leq C\left\Vert
T\right\Vert
{\textstyle\prod\limits_{k=2}^{m}}
\left\Vert \left(  x_{j_{k}}^{(k)}\right)  _{j_{k}=1}^{n}\right\Vert
_{w,\left(  2m-2\right)  ^{\ast}}%
\]
for all continuous $m$-linear forms from $X_{1}\times\cdots\times X_{m}$ to
$\mathbb{K}$ and
\[
C=2^{\frac{(m-1)-1}{2}}.
\]
So let us fix some $a_{1}$ in $X_{1}$ and consider the $\left(  m-1\right)
$-linear form%
\[
R:=T\left(  a_{1},\cdot,...,\cdot\right)  .
\]
We know that $R$ is $\left(  2;\left(  2(m-1)\right)  ^{\ast},....,\left(
2(m-1)\right)  ^{\ast}\right)  $-summing and we want to have a result of the
type $\left(  \mathbf{s};m^{\ast},....,m^{\ast}\right)  $-summing for $R$.
Since for $R$ we have%
\[
\frac{1}{r}-\left\vert \frac{1}{\mathbf{p}}\right\vert +\left\vert \frac
{1}{\mathbf{q}}\right\vert =\frac{1}{2}-\frac{m-1}{\left(  2(m-1)\right)
^{\ast}}+\frac{m-1}{m^{\ast}}>0
\]
we can apply Theorem \ref{rrr} to $R.$ Thus $R$ is $\left(  \mathbf{s}%
;m^{\ast},....,m^{\ast}\right)  $-summing with%
\[
\frac{1}{s_{k}}-\frac{\left(  m-1\right)  -k+1}{m^{\ast}}=\frac{1}{2}%
-\frac{\left(  m-1\right)  -k+1}{\left(  2(m-1)\right)  ^{\ast}}%
\]
i.e., for any $k=2,...,m-1,$ and%
\[
s_{k}=\frac{2m(m-1)}{m+mk-2k}.
\]
We thus conclude that%
\begin{align*}
&  \left(  \sum_{j_{2}=1}^{n}\left(  .....\left(  \sum_{j_{m}=1}^{n}\left\vert
R\left(  x_{j_{2}}^{(2)},\dots,x_{j_{m}}^{(m)}\right)  \right\vert ^{s_{m}%
}\right)  ^{\frac{1}{s_{m}}s_{m-1}}.....\right)  ^{\frac{1}{s_{3}}s_{2}%
}\right)  ^{\frac{1}{s_{2}}}\\
&  \leq2^{\frac{m-2}{2}}\left\Vert R\right\Vert
{\textstyle\prod\limits_{k=2}^{m}}
\left\Vert \left(  x_{j_{k}}^{(k)}\right)  _{j_{k}=1}^{n}\right\Vert
_{w,m^{\ast}}.
\end{align*}
Thus%
\begin{align*}
&  \left(  \sum_{j_{2}=1}^{n}\left(  .....\left(  \sum_{j_{m}=1}^{n}\left\vert
T\left(  a_{1},x_{j_{2}}^{(2)},\dots,x_{j_{m}}^{(m)}\right)  \right\vert
^{s_{m}}\right)  ^{\frac{1}{s_{m}}s_{m-1}}.....\right)  ^{\frac{1}{s_{3}}%
s_{2}}\right)  ^{\frac{1}{s_{2}}}\\
&  \leq2^{\frac{m-2}{2}}\left\Vert T\right\Vert \left\Vert a_{1}\right\Vert
{\textstyle\prod\limits_{k=2}^{m}}
\left\Vert \left(  x_{j_{k}}^{(k)}\right)  _{j_{k}=1}^{n}\right\Vert
_{w,m^{\ast}}%
\end{align*}
for all $a_{1}$ fixed. In other words,%
\begin{align*}
&  \sup_{\left\Vert x_{j_{1}}^{(1)}\right\Vert \leq1}\left(  \sum_{j_{2}%
=1}^{n}\left(  .....\left(  \sum_{j_{m}=1}^{n}\left\vert T\left(  x_{j_{1}%
}^{(1)},\dots,x_{j_{m}}^{(m)}\right)  \right\vert ^{s_{m}}\right)  ^{\frac
{1}{s_{m}}s_{m-1}}.....\right)  ^{\frac{1}{s_{3}}s_{2}}\right)  ^{\frac
{1}{s_{2}}}\\
&  \leq2^{\frac{m-2}{2}}\left\Vert T\right\Vert
{\textstyle\prod\limits_{k=2}^{m}}
\left\Vert \left(  x_{j_{k}}^{(k)}\right)  _{j_{k}=1}^{n}\right\Vert
_{w,m^{\ast}}%
\end{align*}
for
\[
s_{k}=\frac{2m(m-1)}{m+mk-2k}%
\]
and any $m$-linear form from $X_{1}\times\cdots\times X_{m}$ to $\mathbb{K}$.
Since in $\ell_{m}^{n}$ we have $\left\Vert \left(  e_{j}\right)  _{j=1}%
^{n}\right\Vert _{w,m^{\ast}}=1,$ the proof is done.

\section{The proof of the theorem \ref{444}: part 2, on the optimality of the
exponents\label{55}}

Consider $T_{n}:\ell_{m}^{n}\times\cdots\times\ell_{m}^{n}\rightarrow
\mathbb{K}=\mathbb{R}$ or $\mathbb{C}$ given by
\[
T_{n}(x^{(1)},...,x^{(m)})=%
{\textstyle\sum\limits_{j=1}^{n}}
x_{j}^{(1)}..,x_{j}^{(m)}.
\]
Note that%
\[
\left\Vert T_{n}\right\Vert =1
\]
and if there is $C$ such that
\[
\left(  \sum_{j_{1}=1}^{n}\left(  .....\left(  \sum_{j_{m}=1}^{n}\left\vert
T_{n}\left(  e_{j_{1}},\dots,e_{j_{m}}\right)  \right\vert ^{s_{m}}\right)
^{\frac{1}{s_{m}}s_{m-1}}.....\right)  ^{\frac{1}{s_{2}}s_{1}}\right)
^{\frac{1}{s_{1}}}\leq C\left\Vert T_{n}\right\Vert ,
\]
then%
\[
n^{\frac{1}{s_{1}}}\leq C
\]
and since $n$ is arbitrary this means that $s_{1}$ cannot be a positive real
number$.$ So,
\begin{equation}
\sup_{j_{1}}\left(  \sum_{j_{2}=1}^{n}\left(  .....\left(  \sum_{j_{m}=1}%
^{n}\left\vert T\left(  e_{j_{1}},\dots,e_{j_{m}}\right)  \right\vert ^{s_{m}%
}\right)  ^{\frac{1}{s_{m}}s_{m-1}}.....\right)  ^{\frac{1}{s_{3}}s_{2}%
}\right)  ^{\frac{1}{s_{2}}}\leq C\left\Vert T\right\Vert \label{777}%
\end{equation}
for all $m$--linear forms $T:\ell_{m}^{n}\times\cdots\times\ell_{m}%
^{n}\rightarrow\mathbb{K}$ and let us estimate the other exponents. Now
consider%
\[
T_{n-1}(x^{(1)},...,x^{(m)})=x_{1}^{(1)}%
{\textstyle\sum\limits_{j=1}^{n}}
x_{j}^{(2)}...x_{j}^{(m)}.
\]
Note that%
\[
\left\Vert T_{n-1}\right\Vert =n^{1/m}%
\]
Plugging $T_{n-1}$ into (\ref{777}) we have%
\[
n^{1/s_{2}}\leq Cn^{1/m}%
\]
for all $n$, and thus%
\[
s_{2}\geq m
\]
and so on. More precisely, to deal with $s_{3}$ we note that $s_{2}=m$ and
consider%
\[
T_{n-2}(x^{(1)},...,x^{(m)})=x_{1}^{(1)}x_{1}^{(2)}\left(
{\textstyle\sum\limits_{j=1}^{n}}
x_{j}^{(3)}...x_{j}^{(m)}\right)  .
\]
Since,%
\[
\left\Vert T_{n-2}\right\Vert =n^{\frac{2}{m}},
\]
plugging $T_{n-2}$ into (\ref{777}) we have%
\[
n^{1/s_{3}}\leq Cn^{2/m}%
\]
for all $n$, and thus%
\[
s_{3}\geq\frac{m}{2}.
\]
Following this line we show that%
\[
s_{k}\geq\frac{m}{k-1}%
\]
for all $k=3,...,m$.

\section{Back to the origins: Hardy and Littlewood concerns (bilinear case)}

As we mentioned before, the original estimates of Hardy and Littlewood,
\cite{hardy}, dealt with the bilinear case. In \cite{p2} a general formulation
of the bilinear case was presented as follows:

\begin{theorem}
[Pellegrino--Santos--Serrano--Teixeira \cite{p2}]\label{wq} Let $p,q\in
\lbrack2,\infty]$ with $\frac{1}{p}+\frac{1}{q}<1$, and $a,b>0.$ The following
assertions are equivalent:

\begin{itemize}
\item[(a)] There is a constant $C_{p,q,a,b}\geq1$ such that
\begin{equation}
\left(  \sum_{i=1}^{n}\left(  \sum_{j=1}^{n}\left\vert U(e_{i},e_{j}
)\right\vert ^{a}\right)  ^{\frac{1}{a} \cdot b}\right)  ^{\frac{1}{b}}\leq
C_{p,q,a,b}\left\Vert U\right\Vert , \label{811}%
\end{equation}
for all bilinear operators $U \colon\ell_{p}^{n}\times\ell_{q}^{n}
\rightarrow\mathbb{K}$ and all positive integers $n$.

\smallskip

\item[(b)] The exponents $a,b$ satisfy $\left(  a,b\right)  \in\lbrack\frac
{q}{q-1},\infty)\times\lbrack\frac{pq}{pq-p-q},\infty)$ and
\begin{equation}
\frac{1}{a}+\frac{1}{b}\leq\frac{3}{2}-\left(  \frac{1}{p}+\frac{1}{q}\right)
. \label{obey}%
\end{equation}

\end{itemize}
\end{theorem}

When $a,b\in(0,\infty)$ satisfy the item $(b)$ of the previous theorem we call
$(a, b)$ \textit{admissible exponents.}

The above result is essentially final for the bilinear case, but note that the
critical case
\[
\frac{1}{p}+\frac{1}{q}=1
\]
is not considered. This is natural, because this critical case was overlooked
even by Hardy and Littlewood. As a consequence of Theorem \ref{444}, we have:

\begin{corollary}
Let $a,b\in(0,\infty].$ The following assertions are equivalent:

\begin{itemize}
\item[(a)] For all bilinear forms $U\colon\ell_{2}^{n}\times\ell_{2}%
^{n}\rightarrow\mathbb{K}$ and all positive integers $n$ we have
\[
\left(  \sum_{i=1}^{n}\left(  \sum_{j=1}^{n}\left\vert U(e_{i},e_{j}%
)\right\vert ^{a}\right)  ^{\frac{1}{a}\cdot b}\right)  ^{\frac{1}{b}}%
\leq\left\Vert U\right\Vert .
\]

\smallskip

\item[(b)] The exponents $a,b$ satisfy%
\[
b=\infty\text{ and }a\geq2.
\]

\end{itemize}
\end{corollary}

As usual, above, the sum $\left(  \sum_{i=1}^{n}\left(  \bullet\right)
^{b}\right)  ^{1/b}$ is replaced by the $\sup$ norm when $b=\infty.$

The next result shows what happens with the dependence on $n$ when we deal
with non admissible exponents in the critical case. The non critical case was
already investigated in \cite{p2}.

\begin{theorem}
Let $a,b\in(0,\infty].$ If $(a,b)$ are non-admissible then
\[
\left(  \sum_{i=1}^{n_{1}}\left(  \sum_{j=1}^{n_{2}}\left\vert U(e_{i}%
,e_{j})\right\vert ^{a}\right)  ^{\frac{1}{a}\cdot b}\right)  ^{\frac{1}{b}%
}\leq n_{1}^{\frac{1}{b}}\cdot n_{2}^{\frac{1}{a}-\frac{1}{2}}\left\Vert
U\right\Vert ,
\]
for all bilinear forms $U\colon\ell_{2}^{n_{1}}\times\ell_{2}^{n_{2}%
}\rightarrow\mathbb{K}$ and all positive integers $n$. Furthermore, the
exponents $s_{2}:=\frac{1}{a}-\frac{1}{2}$ is sharp and when $s_{2}:=\frac
{1}{a}-\frac{1}{2}$ the exponent $s_{1}=\frac{1}{b}$ is sharp.
\end{theorem}

\begin{proof}
Let $x,y>0$ be such that $\frac{1}{a}=\frac{1}{2}+\frac{1}{x}$ and $\frac
{1}{b}=\frac{1}{\infty}+\frac{1}{y}$. By the H\"{o}lder's inequality for mixed
sums, we have
\begin{align*}
\left(  \sum_{i=1}^{n_{1}}\left(  \sum_{j=1}^{n_{2}}\left\vert U(e_{i}%
,e_{j})\right\vert ^{a}\right)  ^{\frac{1}{a}\cdot b}\right)  ^{\frac{1}{b}}
&  \leq\sup_{i}\left(  \sum_{j=1}^{n_{2}}\left\vert U(e_{i},e_{j})\right\vert
^{2}\right)  ^{\frac{1}{2}}\cdot\left(  \sum_{i=1}^{n_{1}}\left(  \sum
_{j=1}^{n_{2}}\left\vert 1\right\vert ^{x}\right)  ^{\frac{1}{x}\cdot
y}\right)  ^{\frac{1}{y}}\\
&  \leq n_{1}^{\frac{1}{b}}\cdot n_{2}^{\frac{1}{a}-\frac{1}{2}}\left\Vert
U\right\Vert .
\end{align*}
Now let us prove that the exponents are sharp. Define $T_{0}:\ell_{2}^{n_{1}%
}\times\ell_{2}^{n_{2}}\longrightarrow\mathbb{K}$ given by $T_{0}%
(x,y)=x_{1}\sum_{j=1}^{n_{2}}y_{j}$. We have $\Vert T_{0}\Vert=n_{2}^{\frac
{1}{2}}.$ Thus, exist $t_{1},t_{2}\geq0$ such that
\[
\left(  \sum_{i=1}^{n_{1}}\left(  \sum_{j=1}^{n_{2}}\left\vert T_{0}%
(e_{i},e_{j})\right\vert ^{a}\right)  ^{\frac{1}{a}\cdot b}\right)  ^{\frac
{1}{b}}\leq n_{1}^{t_{1}}\cdot n_{2}^{t_{2}}\left\Vert T_{0}\right\Vert
\]
and thus
\[
\left(  \sum_{j=1}^{n_{2}}\left\vert T_{0}(e_{1},e_{j})\right\vert
^{a}\right)  ^{\frac{1}{a}}\leq n_{1}^{t_{1}}\cdot n_{2}^{t_{2}}\cdot
n_{2}^{\frac{1}{2}}.
\]
Thus, for all $n_{1},n_{2}$ we have
\[
n_{2}^{\frac{1}{a}}\leq n_{1}^{t_{1}}\cdot n_{2}^{t_{2}+\frac{1}{2}}%
\]
and we conclude that
\[
t_{2}\geq\frac{1}{a}-\frac{1}{2}.
\]
Now let us fix the optimal exponent $t_{2}=\frac{1}{a}-\frac{1}{2}$.
Considering the bilinear map $A:\ell_{2}^{n}\times\ell_{2}^{n}\longrightarrow
\mathbb{K}$ given by the Kahane--Salem--Zygmund inequality (see \cite{ab}) we
have
\begin{align*}
n^{\frac{1}{a}+\frac{1}{b}} &  =\left(  \sum_{i=1}^{n}\left(  \sum_{j=1}%
^{n}\left\vert A(e_{i},e_{j})\right\vert ^{a}\right)  ^{\frac{1}{a}\cdot
b}\right)  ^{\frac{1}{b}}\leq Cn^{\frac{1}{a}-\frac{1}{2}}n^{t_{1}}\left\Vert
A\right\Vert \\
&  \leq Cn^{\frac{1}{a}-\frac{1}{2}+t_{1}}n^{\frac{1}{2}}.
\end{align*}
Since $n$ is arbitrary we conclude that $t_{1}=\frac{1}{b}$ is optimal.
\end{proof}

\textbf{Acknowledgment} This paper is part of the author's PhD thesis written
under supervision of Professor Joedson Santos.

\end{document}